\documentclass{article}
\usepackage{amsfonts,amsmath, amssymb,latexsym,mathrsfs}
\usepackage[all,cmtip]{xy}
\usepackage{float}
\def\binom#1#2{{#1\choose#2}}
\def\Z{{\mathbb Z}}
\def\Zhat{{\widehat{\Z}}}
\def\A{{\mathbb A}}
\def\Nm{{\rm Nm}}

\def\res{{\rm Res}}
\def\Avg{{\rm Avg}}
\def\SL{{\rm SL}}

\def\SO{{\rm SO}}

\def\sq{{\rm sq}}

\def\cO{{\mathcal O}}
\def\cL{{\mathcal L}}

\def\Disc{{\rm Disc}}

\def\Aut{{\rm Aut}}
\def\irr{{\rm irr}}
\def\gen{{\rm gen}}
\def\red{{\rm red}}
\def\Vol{{\rm Vol}}
\def\R{{\mathbb R}}

\def\FF{{\mathcal F}}
\def\RR{{\mathcal R}}
\def\Q{{\mathbb Q}}

\def\V{{\mathcal V}}
\def\Z{{\mathbb Z}}

\def\Q{{\mathbb Q}}
\def\C{{\mathbb C}}

\newcommand*{\ind}{{\rm ind}}

\newcommand*{\ra}{\rightarrow}

\def\tr{{\rm tr}}

\def\Disc{{\rm Disc}}
\newtheorem{theorem}{Theorem}[section]
\newtheorem{thm}[theorem]{Theorem}
\newtheorem{cor}[theorem]{Corollary}
\newtheorem{lemma}[theorem]{Lemma}

\newtheorem{proposition}[theorem]{Proposition}
\newenvironment{proof}{\noindent {\bf Proof:}}{$\Box$ \vspace{2 ex}}

\usepackage{amscd,amssymb,amsmath}
\usepackage{color}

\author{Arul Shankar and Jacob Tsimerman}
\title{Heuristics for the asymptotics of the number of $S_n$-number fields}
\begin{document}
\maketitle
\begin{abstract}
We give a heuristic argument supporting conjectures of Bhargava on the
asymptotics of the number of $S_n$-number fields having bounded
discriminant.  We then make our arguments rigorous in the case $n=3$
giving a new elementary proof of the Davenport-Heilbronn theorem.  Our
basic method is to count elements of small height in $S_n$-fields
while carefully keeping track of the index of the monogenic ring that
they generate.
\end{abstract}

\section{Introduction}

A classical question in analytic number theory is to count the number
of algebraic number fields of bounded discriminant. To make the
question more precise, for a transitive subgroup $G<S_n$, we define
$N(G,X)$ to be the number of degree-$n$ number fields, with
discriminant bounded by $X$ whose Galois closure has Galois group
$G$. There has been much work on the function $N(G,X)$, both
conjectural and unconditional.  It is conjectured by Malle \cite{Malle}
that $N(G,X)\asymp X^a\ln(X)^b$ with precise values for $a,b$. This
conjecture was shown to be incorrect for certain cases by Kluners
\cite{Kl}, and later corrected by Turkelli \cite{Tu}.

Unconditionally, a stronger version of Malle's conjecture (recovering
the asymptotics of $N(G,X)$, and not merely its growth) is known for
abelian groups $G$ by work of Wright \cite{Wright}. For nilpotent
groups $G$, Kluners--Malle \cite{KM} prove a weak form of Malle's
conjecture, in which upper and lower bounds, differing only by a
factor of $O_\epsilon(X^\epsilon)$, are proved for $N(G,X)$. Versions
of Malle's conjecture are also known for certain products of groups,
and wreath product of groups (see for example \cite{CDO,Kl1,Wang1,LOWW}).

For the important case $G=S_n$, only four cases are known. The case
$n=2$ is trivial. The case $n=3$ is due to work of
Davenport--Heilbronn \cite{DH}, while the cases $n=4,5$ are both
results of Bhargava \cite{dodqf, dodpf}.  In all these cases, the
authors prove that $N(S_n,X)\sim c_nX$ for explicit constants
$c_n$. For general $n$, the best known upper and lower bounds are due
to work of Lemke Oliver and Thorne \cite{LOT}, and Bhargava, Wang, and
the first named author \cite{BSW}, respectively.  Both these results draw
from the methods of previous foundational work by Ellenberg--Venkatesh
\cite{EV}.

In a different direction, Bhargava \cite{Bhamass1} gives a
constant $c_n$ for all integers $n\geq 2$, and conjectures that
$N(S_n,X) \sim c_nX$. The constant $c_n$ is inspired by (and
consistent with) the known results for $n=3$, $4$, and $5$.  The
justification for this conjecture follows from Serre's mass formula
\cite{Serre} together with an assumption that degree-$n$ extensions of
local fields can be independently patched together to form
$S_n$-number fields.

In this paper, we give a different heuristic justification for the
constant $c_n$. In fact, we give a procedure to compute $N(S_n,X)$. At
a certain point our method requires executing a sieve which we are not
able to do in general. In particular, we need to show the independence
of a certain set of thin congruence conditions on lattice points
within a region of Euclidean space. Our inability to show this is the
reason why we do not provably compute $N(S_n,X)$. Nonetheless, in \S3
we execute our method rigorously in the case of $n=3$, recovering the
result of Davenport-Heilbronn.

The way that the results for $n=3,4,5$ have been proven is by finding
a parametrization of the space of rings of rank $n$ over $\Z$ in terms
of orbits of a reductive group acting on a lattice. Asymptotics for
the number of these orbits having bounded discriminant are then
computed using geometry-of-numbers methods. Finally, a sieve is
performed to compute the asymptotics of maximal orders. The main
difficulty in generalizing this approach to counting fields of degree
$n>5$ is the lack of a convenient parametrization for rank $n$ rings.

Our method is to instead count algebraic integers
$\alpha$ of height $\leq Y$ inside every degree-$n$ field $K$ with
bounded discriminant for varying $Y$. With one hiccup, this is fairly
straightforward to do since it essentially amounts to counting the
total number of algebraic numbers $\alpha$ of degree $n$ of height
$\leq Y$ which can be done by looking at the minimal polynlmial of
$\alpha$. On the other hand, if the lattice given by $\cO_K$ is
sufficiently regular then the number of such $\alpha$ is given by
counting points in a lattice and can therefore be well approximated by
the ratio $\frac{C(Y)}{\Disc(K)}$, for some explicit function
$C(Y)$. This then gives a family of identities parametrized by $Y$
from which it is straightforward to recover the asymptotic behaviour
of $N(S_n,X)$.

The hiccup alluded to above is that to recover the discriminant of $K$
from the minimal polynomial of $\alpha$ we need to know the index of
$\Z[\alpha]$ inside the maximal order $\cO_K$. This is given by
independent congruence conditions for each prime, and ends up giving a
thin family of congruence conditions. Proving this independence is the
only part of our argument which remains conditional.

This paper is organized as follows. In \S2 we give the general setup
and the proof that the heuristic assumption \eqref{MHA} implies that
$N(S_n,X)\sim c_n X$. In \S3 we execute the argument unconditionally
in the case of $n=3$.

\section{Heuristics}

We fix a positive integer $d$ and consider the family $\FF_d$ of
degree-$d$ $S_d$-number fields $K$. The purpose of this section is to
prove the following result.
\begin{theorem}\label{thheu}
Assume the Main Heuristic Assumption \eqref{MHA}. Then we have
\begin{equation*}
\#\bigl\{K\in\FF_d:|\Delta(K)|<X\bigr\}\sim
\Bigl(\frac{1}{2}\sum_{\substack{[K:\R]=d}}\frac{1}{\#\Aut(K)}\Bigr)
\prod_p\Bigl(1-\frac{1}{p}\Bigr)
\left(\sum_{\substack{[K:\Q_p]=d}}
\frac{|\Disc(K)|_p}{\#\Aut(K)}\right)X.
\end{equation*}
\end{theorem}


This section is organized as follows. In \S2.1, we define heights and
establish a height preserving bijection between the set of degree-$d$
fields number $K$ such that the normal closure of $K$ has Galois group
$S_d$ over $\Q$ (degree-$d$ $S_d$-number fields) along with an element
of $\cO_K$ with a certain subset of monic integer degree-$d$
polynomials. Then in \S2.2, we make our fundamental heuristic
assumption regarding the asymptotics of the number of such monic
integral polynomials $f(x)$, such that the coefficients of $f$ are
bounded and satisfy certain congruence conditions. This asymptotic is
expressed as a product of local densities. In \S2.3, we compute these
local $p$-adic densities using a Jacobian change of variables and in
\S2.4 we compute the local volume at infinity. Finally, in \S2.5, we
combine our results to prove Theorem \ref{thheu}, recovering
Bhargava's heuristics.

\subsection{The main bijection and setup}


Fix an integer $d\geq 2$. We choose a height function
\begin{equation*}
  h:\bigcup K_\infty\to\R_{\geq 0},
\end{equation*}
where the union is over all degree-$d$ etal\'{e} algebras $K_\infty$
over $\R$, such that $h$ satisfies the following three conditions: (a)
the set of elements with $h=1$ is compact and has measure $0$, (b) $h$
scales linearly, i.e., $h(\lambda x)=\lambda h(x)$ for $\lambda\in\R$
and $x\in \cup K_\infty$, and (c) the function $h$ is nonzero away
from elements $0\in K_\infty$.

Consider a degree-$d$ $S_d$-number field $K\in\FF_d$. Then we use the
natural embedding $\iota: K\to K\otimes\R$ to define a height function
on $K$. Namely, we set $h(\alpha):=h(\iota(\alpha))$.  Let $S_K$
denote the set of elements $\alpha\in\cO_K$ that are {\it reduced},
i.e., have trace in $\{0,1,\ldots,d-1\}$. For a real number $Y>0$, we
let $S_K(Y)$ denote the elements $\alpha\in S_K$ such that
$h(\alpha)<Y$.

Let $V$ denote the space of monic polynomials $f(x)$ of degree $d$,
and let $V_0$ (resp.\ ($V_d$) denote the subspace of $V$ consisting of
elements $f$ with trace $0$ (resp.\ with trace in
$\{0,1,\ldots,d-1\}$). We let $\Delta(f)$ denote the discriminant of
$f$. Given a polynomial $f(x)\in V(\R)$ with nonzero discriminant, we
obtain a pair $(\R[x]/f(x),x)$ of a degree-$d$ etal\'{e} algebra
$K_\infty$ over $\R$, along with an element in $K_\infty$.  We define
the height of a polynomial $f\in V(\R)$ with nonzero discriminant via
$h(f):=h(\alpha)$, where $f$ corresponds to the pair
$(K_\infty,\alpha)$.

Let $V(\Z)^\gen$ denote the subset of $V(\Z)$ consisting of
polynomials $f$ such that $R_f:=\Z[x]/f(x)$ is an order in an
$S_d$-number field, and for subsets $L$ of $V(\Z)$ define
$L^\gen:=L\cap V(\Z)^\gen$. For $f\in V(\Z)^\gen$, we denote by
$\ind(f)$ the index of the order $R_f$ in the maximal order of its
fraction field. By sending an element $\alpha\in K$ to its minimal
polynomial, we obtain a bijection between the set of pairs
$(K,\alpha\in S_K)$ and the set $V_d(\Z)^\gen$. Throughout this
section, we fix a constant $\delta>1$. Keeping track of discriminants
and the index, we obtain the following equality:
\begin{equation}\label{hmain}
  \sum_{\substack{K\in\FF_d\\X<|\Delta(K)|<\delta X}}|S_K(Y)|=
  \sum_{n=1}^\infty \#\{f\in V_d(\Z)^\gen:
  \ind(f)=n,\;h(f)<Y,\;n^2X<|\Delta(f)|<\delta n^2X\}
\end{equation}

Now, by Minkowski's theorem, for any such field $K$, $S_K(Y)$ is
nonempty as long as $Y\gg X^{1/(2d-2)}$. From now on we thus restrict
to $Y=X^{1/(2d-2)+o(1)}$.  Note that for an element $\alpha\in K$ with
$h(\alpha)=H$, the discriminant of $\Z[\alpha]$ is of size at most
$O(H^{d(d-1)})$.  Therefore the index $n$ is at most
$O\Bigl(H^{d(d-1)/2}/|\Delta_K|^{1/2}\Bigr)$, and so the sum in
\eqref{hmain} goes up to $n=O(X^{d/4-1/2+o(1)})$.

\subsection{Local densities and big heuristic assumption}

Let $\sigma(n)$ denote the density in $V_d(\Zhat)$ of those $f\in
V_d(\Zhat)$ such that the index of $\Zhat[x]/f$ in the corresponding
maximal order is exactly $n$. Fix a constant $\delta>1$. We make the
following assumption.

\vspace{.1in}
\noindent \textbf{ Main Heuristic assumption:} On average over $n$, we have
\begin{equation}\label{MHA}
\begin{array}{rcl}
&&\displaystyle\#\Bigl\{f\in
V_d(\Z)^\gen:\ind(f)=n,\;h(f)<Y,\;n^2X<|\Delta(f)|<\delta n^2X\Bigr\}\\[.15in]
&&\sim
\displaystyle\sigma(n)\Vol\Bigl(\bigl\{f\in
V_d(\R):\;h(f)<Y,\;n^2X<|\Delta(f)|<\delta n^2X\bigr\}\Bigr),
\end{array}
\end{equation}
for $Y=X^{1/(2d-2)+o(1)}$.

\vspace{.1in}

Next, we apply the transformation $\theta$ on $V(\R)$, which acts on
$f(x)\in V(\R)$ by dividing all the roots of $f(x)$ by
$Y$. Equivalently, for every $k$, the map $\theta$ scales the
$x^k$-coefficient of $f(x)$ by $1/Y^k$. It is easy to see that we have
$h(\theta\cdot f)=h(f)/Y$ and $\Delta(\theta\cdot
f)=Y^{-d(d-1)}\Delta(f)$.  We consider $V_d(\R)$ as a subset of
$V(\R)$. Applying $\theta$ will map $V_d$ into a union of hyperplanes
in $V(\R)$, namely those having having traces in
$\{0,1/Y,\ldots,d/Y\}$. We may thus write
\begin{align*}
&\Vol\Bigl(\bigl\{f\in V_d(\R):\;h(f)<Y,\;n^2X<|\Delta(f)|<\delta n^2X\bigr\}\Bigr)\\
& \sim dY^{\binom{d+1}{2}-1}\cdot\Vol\Bigl(\bigl\{f\in V_0(\R):\;h(f)<1,\;n^2XY^{-d(d-1)}<|\Delta(f)|<\delta n^2XY^{-d(d-1)}\bigr\}\Bigr).
\end{align*}

We make the following definitions. For complex numbers $s$ where the
sum converges, and for real numbers $t>0$, we define
\begin{equation*}
\begin{array}{rcl}
L(s)&=&\displaystyle\sum_{n\geq 1} \frac{\sigma(n)}{n^s};\\[.1in]
g(t)&=&\displaystyle d\cdot \Vol\Bigl(\bigl\{f\in V_0(\R):\;h(f)<1,\;t^2<|\Delta(f)|<\delta t^2\bigr\}\Bigr).
\end{array}
\end{equation*}
In the next subsection, we will see that $L(s)$ converges absolutely
to the right of $\Re(s)=-1$, and has an analytic continuation to the
left of $\Re(s)=-2$ with a simple pole at $s=-1$. It is easy to see
that $g(t)$ tends to 0 as $t\to 0$, and has compact support.

Next, we set $R=Y^{\binom{d}{2}}/X^{1/2}$.  Assuming the Main
Heuristic Assumption \eqref{MHA}, we see from \eqref{hmain} and the
above discussion that we have
\begin{equation}\label{eqAMHsecond}
\begin{array}{rcl}
\displaystyle\sum_{\substack{K\in\FF_d\\X<|\Delta(K)|<\delta X}}|S_K(Y)|&=&
\displaystyle Y^{\binom{d+1}{2}-1}\sum_{n\geq 1}\sigma(n)g\Bigl(\frac{n}{R}\Bigr)\\[.2in]
&=&\displaystyle Y^{\binom{d+1}{2}-1}\int_{\Re(s)=2} L(s)\tilde{g}(s)R^s\\[.2in]
&\sim&\displaystyle Y^{d-1}X^{1/2}\tilde{g}(-1)\res_{s=-1}L(s).
\end{array}
\end{equation}
In the next two subsections, we compute the residue of $L(s)$ at
$s=-1$, and the value of $\tilde{g}(-1)$, respectively.


\subsection{Computing the residue of $L(s)$}

Let $K_p$ (resp.\ $K_\infty$) be a degree $d$ \'etale extension of
$\Q_p$ (resp.\ $\R$). For $v=p$ or $\infty$, we have a map
\begin{equation}\label{eqphi}
\begin{array}{rcl}
\phi:K_v &\ra& V(K_v)\\[.1in]
\alpha&\mapsto&\Nm(x-\alpha),
\end{array}
\end{equation}
which is $\Aut(K_v)\ra 1$. Moreover, for $p$ a prime number, the image
of $\phi(\cO_{K_p})$ is contained in $V(\Z_p)$.  We fix the
Haar-measure $\nu$ on $V(\Z_p)\cong\Z_p^d$ (resp.\ $V(\R)\cong\R^d$)
normalized so that $\nu(V(\Z_p)=1$ (resp.\ $\nu(V(\R)/V(\Z)=1$).  We
also fix the Haar-measure $\mu$ on $K_v$ normalized so that
$\mu(\cO_{K_p})=1$ when $v=p$ is prime, and normalized to be standard
Euclidean measure, after identifying $\C\equiv\R^2$ via the basis
$\{1,i\}$, when $v=\infty$.
The following lemma relates the measures $\phi^*\nu$ and $\mu$.
\begin{lemma}\label{lemjac}
With the measures $\nu$ and $\mu$ normalized as above, we have
\begin{equation*}
\begin{array}{rcl}
  \phi^*\nu&=&|\Disc(K_p)^{1/2}|_p|\Disc(x)|^{1/2}_p \mu(x)
      \;\;\;\;\mbox{ when }v=p;\\[.1in]
   \phi^*\nu&=&|\Disc(x)|^{1/2} \mu(x)
   \;\;\;\;\;\;\;\;\;\;\;\;\;\;\;\;\;\;\;\;\;\;\;\;\;\;
   \mbox{ when }v=\infty.
\end{array}
\end{equation*}
\end{lemma}
\noindent The above lemma follows directly for $v=\infty$ and
$K_p=\Q_p^d$. The general case may be reduced to this one by tensoring
with a Galois field $M$ containing $K_p$.

\vspace{.1in}

Therefore, we obtain
\begin{equation*}
\begin{array}{rcl}
  \displaystyle\int_{V(\Z_p)} |\ind(f)|_p^s \nu(f) &=&
  \displaystyle\sum_{[K_p:\Q_p]=d}\frac{|\Disc(K_p)|_p^{1/2}}{|\Aut(K_p)|}
  \int_{\cO_{K_p}}
  |\ind(\alpha)|_p^s|\Disc(\alpha)|^{1/2}_p\mu(\alpha)
\\[.2in]&=&\displaystyle\sum_{[K_p:\Q_p]=d}
\frac{|\Disc(K_p)|_p}{|\Aut(K_p)|}\int_{\cO_{K_p}}
|\ind(\alpha)|_p^{s+1}\mu(\alpha).
\end{array}
\end{equation*}
Taking $s=-1$ now yields
\begin{equation*}
\int_{V(\Z_p)} |\ind(f)|_p^{-1} \nu(f) =\sum_{[K_p:\Q_p]=d}
\frac{|\Disc(K_p)|_p}{|\Aut(K_p)|}=1+1/p+O(1/p^2).
\end{equation*}
It thus follows that $L(s)$ has its rightmost pole at $s=-1$ and that this
pole is simple. Moreover, we clearly have
\begin{equation}\label{eqLresidue}
\res_{s=-1}L(s)=
\prod_p\Bigl(1-\frac{1}{p}\Bigr)
\left(\sum_{[K_p:\Q_p]=d}\frac{|\Disc(K_p)|_p}{|\Aut(K_p)|}\right).
\end{equation}
  
\subsection{Computing $\tilde{g}(-1)$}
We start by writing
\begin{equation*}
\begin{array}{rcl}
\tilde{g}(-1) &=&
\displaystyle d\int_0^{\infty}\Vol\Bigl(\bigl\{f\in V_0(\R):\;h(f)<1,\;
t^2<|\Delta(f)|<\delta t^2\bigr\}\Bigr) \frac{dt}{t^2}
\\[.2in]&=&\displaystyle
d\int_0^{\infty}\Vol\Bigl(\bigl\{f\in V_0(\R):\;h(f)<1,\;s<
|\Delta(f)|<\delta s\bigr\}\Bigr) \frac{ds}{2s^{3/2}}\\[.2in]
&=&
\displaystyle d(\sqrt{\delta}-1)
\int_{\substack{f\in V_0(\R)\\h(f)<1}} |\Delta(f)|^{-\frac{1}{2}}\mu(f).
\end{array}
\end{equation*}

Let $K_\infty$ be a fixed degree-$d$ etal\'{e} algebra over $\R$. In
Lemma \ref{lemjac}, we computed a Jacobian change of variables which
applies to the map $\phi:K_\infty\to V(\R)$. This yields the equality
$df=|\Disc(x)|^{1/2}dx$, where $df$ and $dx$ denote the previously
normalized Haar-measures on $V(\R)$ and $K_\infty$, respectively.

The additive group $\R$ acts on $V(\R)$ via linear change of
variables: an element $\lambda$ of $\R$ sends $f(x)$ to
$f(x+\lambda)$. This action clearly preserves the discriminant.
Furthermore, $\R$ acts on $K_\infty$ by addition. It is easy to see
that the map $\phi$ of \eqref{eqphi} respects the action of $\R$ on
$K_\infty$ and $V(\R)$, which is to say that
$\phi(\alpha+\lambda)=\lambda\cdot \phi(\alpha)$.
This action of $\R$ allows us to write
\begin{equation*}
\begin{array}{rcl}
V(\R)&=&\R\times V_0(\R);\\[.1in]
K_\infty&=&\R\times K_\infty^{(\tr=0)}.
\end{array}
\end{equation*}
The map $\phi$ sends an element $(\lambda,f)\in \R\times V_0(\R)$ to
the pair $(\lambda,\phi(f))$. The Jacobian change of variables of the
maps $V(\R)\to\R\times V_0(\R)$ and $K_\infty\to\R\times
K_\infty^{(\tr=0)}$ are easily computed to be $d$ and $1$,
respectively. Denoting the Haar-measures on $V_0(\R)$ by $d_0(f)$ and
on $K_\infty^{(\tr=0)}$ by $d_0\alpha$, we obtain from Lemma
\ref{lemjac} that
\begin{equation*}
d_0(f)=\frac{1}{d}|\Disc(\alpha)|^{1/2}d_0\alpha.
\end{equation*}
Therefore, we have
\begin{equation}\label{eqgtil}
  \tilde{g}(-1)=(\sqrt{\delta}-1)\sum_{K_\infty}\frac{1}{\#\Aut(K_\infty)}
\int_{\substack{\alpha\in K_\infty^{\tr=0}\\h(\alpha)<1}}d_0\alpha.
\end{equation}

Combining \eqref{eqAMHsecond}, \eqref{eqLresidue}, and \eqref{eqgtil},
we obtain the following result.
\begin{theorem}\label{thheusk}
Let $X>0$ be a real number, eventually going to infinity. Fix a
constant $\delta>1$. Let $Y>0$ be a real number such that
$Y=X^{1/(2d-2)+o(1)}$. Conditional on the Main Heuristic Assumption
\eqref{MHA}, we have
\begin{equation*}
\displaystyle\sum_{\substack{K\in\FF_d\\X<|\Delta(K)|<\delta X}}|S_K(Y)|
\sim
(\sqrt{\delta}-1)X^{1/2}Y^{d-1}
\sum_{K_\infty}\frac{1}{\#\Aut(K_\infty)}
\int_{\substack{\alpha\in K_\infty^{\tr=0}\\h(\alpha)<1}}d_0\alpha
\prod_p\Bigl(1-\frac{1}{p}\Bigr)
\left(\sum_{[K_p:\Q_p]=d}\frac{|\Disc(K_p)|_p}{|\Aut(K_p)|}\right).
\end{equation*}
\end{theorem}

\subsection{Cutting off the cusp}

The purpose of this subsection is to deduce Theorem \ref{thheu} from
Theorem \ref{thheusk}. For this, we will need to do two things. First,
we must bound the number of fields $K$ for which $\cO_K$ is ``skewed'',
and as a result the size of $S_K(Y)$ is anomalous.  Second, for the
remainder of fields $K$, we must precisely estimate the average size
of $S_K(Y)$ for the relevant ranges of $Y$. To accomplish these two
goals, we analyze how $S_K(Y)$ behaves using results purely from
lattice theory.

We pick a small constant $\kappa>0$, eventually tending to $0$, a
constant $0<\kappa_1<\kappa$. For $\delta>1$, let $B^{(\delta)}_X$
denote the set of fields $K$ with $X<|\Disc(K)|<\delta X$, such that the
largest vector in a Minkowski basis for $\cO_K$ has length bounded by
$X^{1/(2d-2)+\kappa_1}$.  We let $C^{(\delta)}_X$ denote the set of
fields $K$ with $X<|\Disc(K)|<\delta X$, and such that $K\not\in
B^{(\delta)}_X$.  Then, we have the following facts, which follow
immediately from the theory of Minkowski bases and Minkowski's
theorem:

\begin{itemize}
\item[{\rm (a)}] If $|\Disc(K)|<2X$ then, for some absolute constant $c>0$, we
  have $S_K(cX^{1/(2d-2)})\geq 1$.
  
\item[{\rm (b)}] For $K$ as above, if $Y>cX^{1/(2d-2)}$ and $Z>1$,
  then we have $S_K(YZ)\ll S_K(Y) Z^{d-1}$.
  
\item[{\rm (c)}] For $K\in C_X^{(\delta)}$, we have
  $S_K(cX^{1/(2d-2)+\kappa_1})\ll
  X^{(d-2)\kappa_1}S_K(cX^{1/(2d-2)})$. It thus follows
  that $$S_K(cX^{1/(2d-2)+\kappa})\ll
  X^{(d-1)\kappa-\kappa_1}S_K(cX^{1/(2d-2)}).$$

\item[{\rm (d)}] For $K\in B_X^{(\delta)}$, we
  have
\begin{equation*}
  S_K(X^{1/(2d-2)+\kappa})=
  \frac{X^{\frac{1}{2}+(d-1)\kappa}}{|\Disc(K)|^{1/2}}
  \int_{\substack{\alpha\in K_\infty^{\tr=0}\\h(\alpha)<1}}d_0\alpha.
\end{equation*}
\end{itemize}

\subsubsection*{Bounding the sum over fields in $C_X^{(\delta)}$}
We have the following lemma regarding fields in $C_X^{(\delta)}$.
\begin{lemma}\label{lemhskew}
Let notation be as above. Then, conditional on the Main Heuristic
Assumption, we have
\begin{equation*}
|C_X^{(\delta)}|\ll X^{1-\kappa_1};\quad\quad\quad
\sum_{K\in C_X^{(\delta)}}|S_K(X^{1/(2d-2)+\kappa})|\ll X^{1+(d-1)\kappa-\kappa_1}.  
\end{equation*}
\end{lemma}
\begin{proof}
By the theory of Minkowski bases, for $K\in C_X^{(\delta)}$ we have
$S_K(X^{1/(2d-2)})\gg X^{\kappa_1}$. Thus, the first assertion of the
lemma follows from Theorem \ref{thheusk} by setting $Y=X^{1/(2d-2)}$.
For the second assertion, note that from Fact (c) above, we have
\begin{equation*}
  \sum_{K\in C_X^{(\delta)}}|S_K(X^{1/(2d-2)+\kappa})| \ll
  X^{(d-1)\kappa-\kappa_1}\sum_{K\in C_X^{(\delta)}}|S_K(cX^{1/(2d-2)})|\ll
X^{1+(d-1)\kappa-\kappa_1},
\end{equation*}
as necessary.
\end{proof}

\subsubsection*{Estimating the size of $B_X^{(\delta)}$}
We fix a signature $\sigma$ at infinite corresponding to the algebra
$K_\infty$ over $\R$. Given a set $F$ of degree-$d$ fields $K$, we let
$F^{(\sigma)}$ denote the subset of fields $K\in F$ such that
$K\otimes\R\cong K_\infty$.
For ease of notation, we define
\begin{equation*}
M=\frac{1}{\#\Aut(K_\infty)}
\prod_p\Bigl(1-\frac{1}{p}\Bigr)
\left(\sum_{[K_p:\Q_p]=d}\frac{|\Disc(K_p)|_p}{|\Aut(K_p)|}\right).
\end{equation*}

From Theorem \ref{thheusk}, we have
\begin{equation*}
\displaystyle\sum_{\substack{K\in\FF_d^{(\sigma)}\\X<|\Delta(K)|<\delta X}}|
S_K(X^{1/(2d-2)+\kappa})|
\sim (\sqrt{\delta}-1)M \int_{\substack{\alpha\in K_\infty^{\tr=0}\\h(\alpha)<1}}d_0\alpha
\cdot X^{1+(d-1)\kappa}.
\end{equation*}
Thus, using Fact (d) in conjunction with Lemma \ref{lemhskew}, we
obtain
\begin{equation}\label{eqhbfin}
\begin{array}{rcl}
\displaystyle
\sum_{K\in B^{(\delta),\sigma}_X}\sqrt{\frac{X}{|\Disc(K)|}}
&\sim&
\displaystyle
\Bigl(\int_{\substack{\alpha\in K_\infty^{\tr=0}\\h(\alpha)<1}}d_0\alpha\Bigr)^{-1}
\sum_{\substack{K\in\FF_d^{(\sigma)}\\X<|\Delta(K)|<\delta X}}
|S_K(X^{1/(2d-2)+\kappa})|\\[.4in]
&\sim&\displaystyle (\sqrt{\delta}-1)M\cdot X.
\end{array}
\end{equation}

\subsubsection*{Conclusion}
Set $\delta=1+\epsilon$, where $\epsilon$ will eventually tend to
$0$. From \eqref{eqhbfin} and Lemma \ref{lemhskew}, we have
$$
\sum_{\substack{K\in\FF_d^{(\sigma)}\\X<|\Delta(K)|<\delta X}}
\sqrt{\frac{X}{|\Disc(K)|}}
\sim\frac{\epsilon}{2}MX.
$$
Summing over the $\epsilon$-adic ranges in $[1,X]$, we obtain
\begin{equation*}
\sum_{\substack{K\in\FF_d^{(\sigma)}\\1<|\Delta(K)|<X}}t_K\sim
\frac{1}{1+\epsilon}\frac{M}{2}\cdot X,  
\end{equation*}
where $1/\sqrt{\delta}\leq t_K\leq 1$. Now letting $\epsilon\to 0$
yields Theorem \ref{thheu}.
%
%
%
%
%
%

\subsection{Remarks}

\begin{enumerate}

\item The heuristic above is flexible enough to accomodate finitely
  many local conditions on the fields $K$. Indeed, the archemedian
  places conditions are already accommodated by the height function
  $h$. If we want to impose conditions on the Etale algebra $K_p$ at a
  finite $p$, we may simply record that condition on $f$ when making
  the main bijection, and it will only affect the density computation
  in \S2.3.

\item One may ask a more precise equidistribtion question by asking
  about the shape of the lattice $\cO_K^{\tr=0}$,or even better by
  asking about the distribution of the co-volume 1 lattice $|\Disc
  K|^{1/(2d-2)}\cO_K^{\tr=0}$ inside the space of all covolume 1
  lattices in $K_{\infty}^{\tr=0}$ .The natural guess is that it is
  equidistributed with respect to Haar measure on
  $\SL_{\R}(K_{\infty}^{\tr=0})$, and this is proven modulo an
  $\SO$-action by Bhargava-Harron in the cases $d=3,4,5$. By varying
  the height function gives in our heuristic one obtains a family of
  test functions for the resulting measure, but it appears to the
  authors to be insufficient to determine the measure
  completely. However, the heuristic does recover the distribution on
  the theta functions of the resulting lattices.

\item For the case of $d\geq 6$, our main heuristic really requires an
  average over $n$, since there is $O(1)$ expected points for each $n$
  when $d=6$ and fewer than $1$ expected points when $d>6$.
 
\end{enumerate}

\section{The number of cubic fields having bounded discriminant}

Consider a cubic field $K$ over $\Q$ with ring of integers $\cO_K$, and
discriminant $\Delta(K)$. We say that an element $\alpha\in\cO_K$ is
{\it reduced} if the trace of $\alpha$ is $-1,0,$ or $1$. Define
$|\alpha|_{\infty}=\max_{v\mid\infty} |\alpha|_v$, and for a
real number $Y>0$, let $S_K(Y)$ to be the set of reduced elements
$\alpha\in\cO_K\backslash\Z$ satisfying $|\alpha|_\infty<Y$. For a ring
$R$, let $V(R)$ denote the set of monic cubic polynomials
$f(x)=x^3+tx^2+Ax+B$, where $t\in\{-1,0,1\}$ and $A,B\in R$. We denote
the discriminant of $f(x)$ by $\Delta(f)$. Define the height
functon
\begin{equation*}
\begin{array}{rcl}
 h:V(\R)&\to& \R\\[.05in]
 h(f)&:=&\max |\alpha|,
\end{array}
\end{equation*}
where the maximum is taken over the roots of $f$. We then have the
following lemma whose proof is immediate.

\begin{lemma}\label{lembij}
There is a bijection between the following two sets:
\begin{itemize}
\item[{\rm (1)}] The set of pairs $(K,\alpha)$, where $K$ is a cubic
  field $($up to isomorphism$)$, $\alpha\in S_K(Y)$, and $\Delta(K)<X$.
\item[{\rm (2)}] The set of irreducible polynomials $f(x)\in V(\Z)$
  such that $h(f)<Y$ and $\Delta(\Q[x]/f(x))<X$.
\end{itemize}
\end{lemma}

For a subset $L$ of $V(\Z)$, we denote the set of irreducible elements
in $L$ by $L^\irr$. Given $f\in V(\Z)$ (resp.\ $V(\Z_p)$) with
$\Delta(f)\neq 0$, we define $\ind(f)$ to be the index of $\Z[x]/f(x)$
(resp.\ $\Z_p[x]/f(x)$) in the ring of integral elements in
$\Q[x]/f(x)$ (resp.\ $\Q_p[x]/f(x)$). We then have the following
consequence of Lemma \ref{lembij}.
\begin{equation}\label{eqbijection}
  \sum_{\substack{[K:\Q]=3\\X<|\Delta(K)|< \delta X}} |S_K(Y)|=\sum_{n\geq 1}
  \#\bigl\{f\in V(\Z)^\irr\,:\,\ind(f)=n,h(f)<Y,
  n^2X<|\Delta(f)|<\delta n^2X\bigr\},
\end{equation}
for a fixed constant $\delta>1$.

Recall from \S2 that we denote the density of the set of elements
$f\in V(\Z)$ with index $n$ by $\sigma(n)$. Let $V(\R)_{n^2X,Y}$
denote the set of elements $f(x)\in V(\R)$ such that $h(f)<Y$ and
$n^2X<|\Delta(f)|<\delta n^2X$. Then the main result of this section
is as follows.

\begin{theorem}\label{theorem3}
For some sufficiently small $\kappa>0$, set $Y=X^{1/4+\kappa}$. Then
we have
\begin{equation}\label{eqthm3}
\sum_{\substack{[K:\Q]=3\\X<|\Delta(K)|< \delta X}} |S_K(Y)|=\sum_{n\geq 1}
\sigma(n)\Vol(V(\R)_{n^2X,Y})+o(X).
\end{equation}
\end{theorem}

In conjunction with the results of \S2, Theorem \ref{theorem3}
immediately recovers the Davenport--Heilbronn result on the density of
discriminants of cubic fields.

\begin{theorem}[\cite{DH}]
Let $N^\pm_3(X)$ denote the number of cubic fields $K$ such that
$0<\pm\Delta(K)<X$. Then
\begin{equation*}
N^+_3(X)\sim \frac{1}{12\zeta(3)}X; \quad\quad\quad N^-_3(X)\sim
\frac{1}{12\zeta(3)}X.
\end{equation*}
\end{theorem}

This section is organized as follows. First, in \S3.1, we prove a
variety of estimates and bounds on sets of elements in $V(\Z)$
satisfying various height, discriminant, and index conditions. Then in
\S3.2, we provide an upper bound for the left hand side of Theorem
\ref{eqthm3}, which is optimal up to a factor of
$O_\epsilon(X^{\epsilon})$. Finally, in \S3.3, we execute an inclusion
exclusion sieve to prove Theorem \ref{theorem3} using the counting
results of the previous two subsections.

\subsection{Counting non-maximal integer monic cubic polynomials}

To estimate the number of lattice points in the bounded subsets of
$V(\R)$, we need the following proposition due to
Davenport~\cite{Davenport1}.

\begin{proposition}\label{davlem}
  Let $\RR$ be a bounded, semi-algebraic multiset in $\R^n$
  having maximum multiplicity $m$, and that is defined by at most $k$
  polynomial inequalities each having degree at most $\ell$.  
  Then the number of integral lattice points $($counted with
  multiplicity$)$ contained in the region $\RR$ is
\[\Vol(\RR)+ O(\max\{\Vol(\bar{\RR}),1\}),\]
where $\Vol(\bar{\RR})$ denotes the greatest $d$-dimensional 
volume of any projection of $\RR$ onto a coordinate subspace
obtained by equating $n-d$ coordinates to zero, where 
$d$ takes all values from
$1$ to $n-1$.  The implied constant in the second summand depends
only on $n$, $m$, $k$, and $\ell$.
\end{proposition}

\subsubsection*{Congruence conditions on polynomials with index
  divisible by an integer $n$}

Let $p$ be a prime. The following criterion for $f\in V(\Z)$ with
$\Delta(f)\neq 0$ to have index divisible by a prime $p$ follows
immediately from \cite[Theorem 14]{BST} (originally due to work of
Davenport--Heilbronn \cite{DH}).

\begin{lemma}\label{lemnonmax}
An element $f(x)\in V(\Z)^\irr$ has index divisible by a prime $p$ if
and only if there exists $\bar{r}\in \Z/p\Z$ such that for every lift
$r\in\Z$ of $\bar{r}$, we have $p^2\mid f(r)$ and $p\mid f'(r)$, where
$f'(x)$ is the derivative of $f(x)$.
\end{lemma}

Let $f(x)=x^3+kx^2+Ax+B$ be an element of $V(\Z)$. It follows from
Lemma \ref{lemnonmax} that the residue classes of $A$ and $k$ modulo
$p$ and the residue class of $B$ modulo $p^2$ determine whether or not
$p\mid \ind(f)$. More precisely, we have the following immediate
consequence of Lemma \ref{lemnonmax}.

\begin{cor}\label{corA3}
Let $p$ be a fixed prime and let $k$ and $A$ be fixed integers. The
number of $\bar{B}\in\Z/p^2\Z$ such that $p\mid\ind(x^3+kx^2+Ax+B)$,
for lifts $B\in\Z$ of $\bar{B}$, is determined by the residue classes
of $k$ and $A$ in $\Z/p\Z$. In fact, the number of such
$\bar{B}\in\Z/p^2\Z$ is equal to the number of roots modulo $p$ of
$3x^2+2kx+A$.
\end{cor}

For a positive integer $n$, let $\Sigma_n$ denote the set of
polynomials $f(x)$ in $V(\Z)$ with nonzero discriminant such that
$n\mid\ind(f)$. We will need a version of Corollary \ref{corA3} for
arbitrary integers $n$. To analyze the case when $n$ is divisible by a
prime power $p^\ell$ with $\ell\geq 2$, we write the set
$\Sigma_{p^\ell}$ as the disjoint union
\begin{equation*}
\Sigma_{p^\ell}=\Sigma_{p^\ell}^{(1)}\cup\Sigma_{p^\ell}^{(2)},
\end{equation*}
where $\Sigma_{p^\ell}^{(1)}$ denotes the set of elements $f(x)\in
\V_{p^\ell}(k,A)$ such that the image of $x$ in $R_f:=\Z[x]/f(x)$ is
not a multiple of $p$ in $R_f\backslash\Z$. Then we have the following
lemma.

\begin{lemma}\label{lempropnmpp}
An element $f(x)\in V(\Z)^\irr$ belongs to $\Sigma_{p^\ell}^{(1)}$ if
and only if there exists $\bar{r}\in\Z/p^\ell\Z$ such that for every
lift $r\in\Z$ of $\bar{r}$, we have $\frac{1}{4}p^{2\ell}\mid f(r)$
and $\frac{1}{2}p^\ell\mid f'(r)$.
\end{lemma}
\begin{proof}
We start with the case when the splitting type of $f$ at $p$ is
$(1^21)$.  Let $\sigma_1$, $\sigma_2$, and $\theta$ denote the roots
of $f$ in $\overline{\Q}_p$, with
$\sigma_1\equiv\sigma_2\pmod{p}$. Then $\theta$ belongs to $\Z_p$ and
$\sigma_1$ and $\sigma_2$ either belong to $\Z_p$ and are congruent
modulo $p$, or are conjugate elements in the ring of integers of
integers of a ramified extension of $\Q_p$. For
$r=(\sigma_1+\sigma_2)/2$, the roots of $f(x+r)$ are
$\theta'=\theta-r$, $\sigma_1'=\sigma_1-r$, and
$\sigma_2'=\sigma_2-r=-\sigma_1'$. The $p$-part of the discriminant
$\Delta(f)=\Delta(f(x-r))$ is then computed to be equal to the
$p$-part of $(\sigma_1'-\sigma_2')^2=4\sigma_1'^2$. It follows that
$\frac{1}{2}p^\ell$ divides $\sigma_1'$ and $\sigma_2'$, and therefore
that $f(x+r)$ is of the form $x^3+ax^2+\frac{1}{2}p^\ell
bx+\frac{1}{4}p^{2\ell}c$. Clearly, this remains true for all $r_1$
that are congruent to $r$ modulo $p^\ell$. Evaluating $f(x)$ and
$f'(x)$ at $r$ yields the lemma for this case.

Next assume that the splitting type of $f$ is $(1^3)$. For an
appropriate integer $r$, replace $f(x)$ by $f(x+r)$ such that the
triple zero of $f$ modulo $p$ is at $0$, i.e., $p\mid a$, $p\mid b$,
and $p\mid c$. Since $p\mid\ind(f)$ we must have $p^2\mid c$. Since
$\alpha$ is not a multiple of $p$, it cannot simultaneously happen
that $p^2\mid b$ and $p^3\mid c$. Therefore the only possibilities are
that $p\parallel b$, $p^2\mid c$ or that $p^2\mid b$ and $p^2\parallel
c$. In either case, it is easy to check that the $p$-part of the index
of $f$ is $p$ and that therefore $\ell=1$. The result now follows from
Lemma \ref{lemnonmax}.
%
%
\end{proof}

For a set $S\subset V(\Z)$, let $\nu(S)$ denote the volume of the
closure of $S$ in $V(\hat{\Z})$. For integers $k$ and $A$, let
$S(k,A)$ denote the set of integers $B$ such that $x^3+kx^2+Ax+B$
belongs to $S$. Let $\nu(k,A;S)$ denote the volume of the closure of
$S(k,A)$ in $\Z_p$. Here, we computes volumes in $V(\Z_p)$ and $\Z_p$
in terms of Euclidean measure normalized so that $V(\Z_p)$ and $\Z_p$,
respectively, have volume $1$. Then we have the following result.

\begin{proposition}\label{propnonmaxpp}
Let $n$ be a positive integer and write $n=q^3m$, where $m$ is
cube-free. Then
\begin{itemize}
\item[{\rm (a)}] The set $\Sigma_n$ is defined via congruence
  conditions modulo $n^2$.
\item[{\rm (b)}] For $k,A\in\Z$, the density of
  $\Sigma_n(k,A)$ depends only on the congruence classes of $k$ and
  $A$ modulo $n$.
\item[{\rm (c)}] We have the bound
  $\nu(\Sigma_n)\ll_\epsilon\displaystyle\frac{n^\epsilon}{q^5m^2}.$
\end{itemize}
\end{proposition}

\begin{proof}
By the Chinese Remainder Theorem, we may assume that $n=p^\ell$ is a
prime power. We proceed by induction on $\ell$. For $\ell=1$ all three
claims follow from Lemma~\ref{lemnonmax}.  Furthermore, for all
$\ell\geq 1$, the claims of the proposition with $\Sigma_{p^\ell}$
replaced by $\Sigma^{(1)}_{p^\ell}$, follow from
Lemma~\ref{lempropnmpp}.  This also yields the required results for
$\ell=2$, since $\Sigma_{p^2}=\Sigma_{p^2}^{(1)}$.

We now assume that $\ell\geq 3$ and prove the claims of the
proposition with $\Sigma_{p^\ell}$ replaced with
$\Sigma_{p^\ell}^{(2)}$. To this end, let $f(x)=x^3+kx^2+Ax+B$ be an
element of $\Sigma_{p^\ell}^{(2)}$ and let $\alpha$ denote the image
of $x$ in $\Z[x]/f(x)$. Then there exists $r\in\Z$, defined uniquely
modulo $p$, such that $\alpha-r$ is a multiple of $p$.  The polynomial
$f(x+r)=x^3+k'x^2+A'x+B'$ satisfies $p\mid k'$, $p^2\mid A'$ and
$p^3\mid B'$. Furthermore, we have
$$x^3+(k'/p)x^2+(A'/p^2)x+(B'/p^3)\in \Sigma_{p^{\ell-3}}.$$ Parts (a)
and (b) of the proposition follow immediately from induction. Part (c)
also follows since for each fixed $r$, the volume of the corresponding
subset of cubic polynomials with index divisible by $p^\ell$ is
bounded by $O(p^{-6}\cdot \nu(\Sigma_{p^{\ell-3}}))$. Since $r$ is
defined modulo $p$, Part (c), and the proposition, follows.
\end{proof}

\subsubsection*{Bounds on reducible elements}

For squarefree integers $n$, we obtain a bound on the number of
reducible polynomials in $\Sigma_n$ having bounded height.
For a set $S\subset V(\Z)$, let $S^\red$ denote the set of reducible
polynomials in $S$. We have the following result.
\begin{proposition}\label{propred3}
  Let $n$ be a positive squarefree integer. Then
  \begin{equation*}
    \#\{f\in\Sigma_n^\red:h(f)<Y\}\ll_\epsilon
Y^3/n^{1-\epsilon}+n^\epsilon Y.
  \end{equation*}
\end{proposition}
\begin{proof}
  If $f\in V(\Z)$ is reducible, then there exists $r\in\Z$ such that
  $f(r)=0$. Hence we have
  \begin{equation*}
    f(x)=(x-r)(x^2+(r+k)x+b)=x^3+kx^2+(b-r(r+k))x-br
  \end{equation*}
  with $k\in\{-1,0,1\}$.  For such an element $f$ with $h(f)<Y$, it
  follows that $|r|\ll Y$ and $|b| \ll Y^2$. Note that for fixed $r$
  and $k$, the polynomial $\Delta_{r,k}(b):=\Delta(f)$ is a cubic
  polynomial in $b$ with leading coefficient $4$, and in particular,
  that it is nonzero.

  
  Let $f\in\Sigma_n$ be a polynomial with fixed $r$ and $k$. It
  follows that $n^2\mid\Delta(f)$ and that therefore the residue of
  $b$ modulo $n$ has $O(n^\epsilon)$ choices. The proposition now
  follows from the bounds on $|r|$ and $|b|$.
\end{proof}

\subsubsection*{Estimates and bounds on irreducible elements}

Let $V(\R)^\pm_{X,Y}$ denote the set of elements $f\in V(\R)^\pm$ such
that $|\Delta(f)|<X$ and $h(f)<Y$.  We start by estimating the number
of elements in $\Sigma_n$ with bounded height and discriminant, for
squarefree integers $n$.

\begin{theorem}\label{thcount}
Let $m$ be a positive integer and let $n$ be a positive squarefree
integer relatively prime to $m$. Let $\cL\subset V(\Z)$ be a set
defined by congruence conditions modulo $m$. Then
\begin{equation*}
  \displaystyle\#\{f\in\cL\cap\Sigma_n:0<\pm\Delta(f)<X,\;h(f)<Y\}=
  \displaystyle
  \nu(L)\nu(\Sigma_n)\Vol(V(\R)^\pm_{XY})
  +O\left(Y^3m/n^{1-\epsilon}+Y^2mn^{\epsilon}\right).
\end{equation*}
\end{theorem}
\begin{proof}
Given $A\in\Z$, let $R_A$ denote the set of polynomials $f(x)\in
V(\R)^\pm_{XY}$ with $x$-coefficient equal to $A$. Let $\nu_{m,n,A}$
denote the density of the set of polynomials $f(x)\in L\cap\Sigma_n$
whose $x$-coefficient is $A$ within the set of polynomials $f\in
V(\Z)$ whose $x$-coefficient is $A$. From Corollary \ref{corA3} it
follows that $\nu_{m,n,A}$ depends only on the residue of $A$ modulo
$mn$, and that $\nu_{m,n,A}\ll 1/n^{2-\epsilon}$.  Fibering by $A$, we
obtain
\begin{equation}\label{eqcc1}
  \begin{array}{rcl}
  \displaystyle\#\{f\in\cL\cap\Sigma_n:|\Delta(f)|<X,\;h(f)<Y\}&=&
  \displaystyle \sum_{|A|\leq 3Y^2}\#\{R_A\cap L\cap\Sigma_n\}\\[.1in]
  &=& \displaystyle\sum_{|A|\leq 3Y^2}\nu_{m,n,A}|R_A|+O(Y^2mn^\epsilon),
  \end{array}
\end{equation}
where $|R_A|$ denotes the length of $R_A$. Note that we have $|R_A|\ll
Y^3$ from the height bound. Now the average value of $\nu_{m,n,A}$, as
$A$ varies over a complete residue system modulo $mn$, is clearly
equal to $\nu(L)\nu(\Sigma_n)$.

Consider the main term of the second line of \eqref{eqcc1}.  We break
up the sum over arithmetic progressions modulo $mn$. From Proposition
\ref{davlem}, we obtain

\begin{equation*}
\begin{array}{rcl}
 \displaystyle\sum_{|A|\leq 3Y^2}\nu_{m,n,A}|R_A|&=&
 \displaystyle\sum_{d\in\Z/(mn)}\nu_{m,n,d}
 \sum_{\substack{|A|\leq 3Y^2\\A\equiv d\;{\rm mod}\;{mn}}}
 |R_A|
 \\[.3in] &=&
 \displaystyle\sum_{d\in\Z/(mn)}\nu_{m,n,d}\cdot\Bigl(
 \frac{\Vol(V(\R)_{X,Y})}{mn} +
 O\bigl(Y^3\bigr)\Bigr)
 \\[.2in] &=&
 \displaystyle
 \nu_m(L)\nu_n(\Sigma_n)\Vol(V(\R)_{X,Y})+
 O\Bigl(\frac{mY^3}{n^{1-\epsilon}}\Bigr),
\end{array}
\end{equation*}
where we use the fact that $\nu_{m,n,A}\ll_\epsilon
1/n^{2-\epsilon}$. This concludes the proof of the theorem.
\end{proof}

Finally, we prove a bound on the number of polynomials $f$ such that
$\ind(f)$ is divisible by arbitirary positive integers $n$.

\begin{theorem}\label{thbound}
  Let $n$ be a positive integer. We have
  \begin{equation*}
    \#\{f\in\Sigma_n:|\Delta(f)|<X,\;h(f)<Y\}\ll \nu(\Sigma_n)\Bigl(Y^2+n\Bigr)
    \min\{Y^3,X^{\frac12}\}+nY^2.
  \end{equation*}
\end{theorem}
\begin{proof}
Let $R_A$ be defined as in the proof Theorem \ref{thcount}, and note
that $|R_A|\ll\min(Y^3,X^{\frac12})$. Fibering over $k$ and $A$, we
obtain
\begin{align*}
\{f\in\Sigma_n:|\Delta(f)|<X,\;h(f)<Y\}&\ll \sum_{k\in\{-1,0,1\}}\sum_{A<3Y^2}
\bigl(\nu(k,A;\Sigma_n(k,A))\cdot|R_A|+O(n)\bigr)\\
&\ll nY^2+\sum_{k\in\{-1,0,1\}}\sum_{A<3Y^2} \nu(k,A;\Sigma_n(k,A))(\min\{Y^3,X^{\frac12}\})\\
&\ll nY^2+(Y^2+n)\min\{Y^3,X^{\frac12}\}\Avg(\nu(k,A;\Sigma_n(k,A))),
\end{align*}
where the average is over $k\in\{-1,0,1\}$ and $A$ modulo $n$. Since
this average is equal to $\nu(\Sigma_n)$, the theorem follows.
\end{proof}


 

\subsection{An upper bound}
We fix a constant $C>1$ such that for every cubic field $K$, the set
$S_K(CX^{1/4})$ is nonempty. Let $X$ and $Y$ be positive real numbers
such that $Y\geq CX^{1/4}$. Our goal in this section is to prove an
upper bound for number of cubic fields $K$ with discriminant bounded
by $X$, where each field $K$ is counted with weight $|S_K(Y)|$. We
start with the following important lemma:
\begin{lemma}\label{lemnumel}
Let $X$ and $Y$ be as above. Let $K$ be a cubic field such that
$X/2\leq|\Delta(K)|\leq X$. Then we have
\begin{equation*}
\#S_K(Y)/\#S_K(CX^{1/4})\ll Y^2/X^{1/2},
\end{equation*}
where the implied constant is independent of $X$, $Y$, and $K$.
\end{lemma}
\begin{proof}
  We start by picking a Minkowski basis $\langle1,\alpha,\beta\rangle$
  for $\cO_K$. Let $\delta_1$ and $\delta_2$ be such that
  $|\alpha|_\infty=X^{\delta_1}$ and
  $|\beta|_\infty=X^{\delta_2}$. Assume without loss of generality
  that $\delta_1\leq\delta_2$. From our assumption on $C$, it follows
  that $|\alpha|_\infty<CX^{1/4}$. We have
\begin{equation*}
\begin{array}{rcl}
\displaystyle\#S_K(Y)&\asymp&
\displaystyle\frac{Y}{X^{\delta_1}}\cdot
\max\Bigl\{\frac{Y}{X^{\delta_2}},1\Bigr\};\\[.2in]
\displaystyle\#S_K(CX^{1/4})&\asymp&
\displaystyle\frac{CX^{1/4}}{X^{\delta_1}}\cdot
\max\Bigl\{\frac{CX^{1/4}}{X^{\delta_2}},1\Bigr\}.
\end{array}
\end{equation*}
The proof now follows from the fact that $X^{\delta_2}\gg X^{1/4}$.
\end{proof}

We now prove the following crucial upper bound:
\begin{thm}\label{thupbound}
  Let $X$ and $Y$ be as above. Then
  \begin{equation*}
    \sum_{\substack{[K:\Q]=3\\ |\Delta(K)|\leq X}} |S_K(Y)|\ll_{\epsilon} X^{1/2+\epsilon}Y^2,
  \end{equation*}
  where the implied constant only depends on $C$.
\end{thm}
\begin{proof}
We start by counting cubic fields whose discriminants are in a dyadic
range of $M<X$. If $K$ is such a field and if $\alpha\in
S_K(CM^{1/4})$, then any polynomial $f$ corresponding to $(K,\alpha)$
under Lemma \ref{lembij} must satisfy $\ind(f)\ll M^{1/4}$ (since
$\Delta(f)\ll M^{3/2}$). From Lemma \ref{lemnumel} and
\eqref{eqbijection}, we obtain
\begin{equation*}
  \begin{array}{rcl}
    \displaystyle\sum_{\substack{[K:\Q]=3\\ M/2\leq |\Delta(K)|\leq M}} |S_K(Y)|&\ll&
    \displaystyle\frac{Y^2}{M^{1/2}} \sum_{\substack{[K:\Q]=3\\ M/2\leq
        |\Delta(K)|\leq M}} |S_K(CM^{1/4})|\\[.25in]
    &\ll&\displaystyle \frac{Y^2}{M^{1/2}}  \sum_{n\ll M^{1/4}}
    \#\{f\in \Sigma_n\,:\,h(f)<CM^{1/4},|\Delta(f)|<n^2M\}\\[.25in]
 &\ll&\displaystyle \frac{Y^2}{M^{1/2}}\displaystyle\sum_{n\ll M^{1/4}}\Bigl( nM^{1/2} + \nu(n)M^{1/2}\min\{M^{3/4},nM^{\frac12}\}\Bigr)\\[.2in]
 &\ll&Y^2M^{1/2} +Y^2M^{1/2}\displaystyle\sum_{n\ll M^{1/4}} \nu(n)n\\[.2in]
 &\ll&M^{1/2+\epsilon}Y^2,
  \end{array}
\end{equation*}
where the third estimate follows from Theorem \ref{thbound}, and the
last estimate follows from Proposition \ref{propnonmaxpp}.  Summing
$M<X$ over powers of $2$ yields the theorem.
\end{proof}

\subsection{The sieve}
Let $\kappa$, $\delta_1$, and $\delta_2$ be positive real numbers to
be chosen later. Fix $C$ as in the previous subsection. Throughout
this section, we set $Y=CX^{1/4+\kappa}$. We apply the inclusion
exclusion sieve to obtain
\begin{equation}\label{eqsm}
\begin{array}{rcl}
\displaystyle\sum_{\substack{[K:\Q=3]\\|\Delta(K)|\leq X}}|S_K(Y)|
&=&\displaystyle\sum_{n\geq 1}
\#\{f\in V(\Z)^\irr\,:\,\ind(f)=n,h(f)<Y,|\Delta(f)|<n^2X\}\\[.25in]
&=&\displaystyle\sum_{n\geq 1}\sum_{d\geq 1}\mu(d)
\#\{f\in \Sigma_{dn}^\irr\,:\,h(f)<Y,|\Delta(f)|<n^2X\}.
\end{array}
\end{equation}
The next result bounds the tail of the above sum.

\begin{lemma}\label{lemtailestimate}
We have
\begin{equation*}
\sum_{\substack{n,d\geq 1\\nd> X^{1/4+\delta_1}}}
\#\{f\in \Sigma_{dn}^\irr\,:\,h(f)<Y\}\ll_\epsilon
X^{1+6\kappa-2\delta_1+\epsilon}.
\end{equation*}
\end{lemma}
\begin{proof}
Let $f\in V(\Z)^\irr$ be such that $\ind(f)>X^{1/4+\delta_1}$ and
$h(f)<Y$. Denote $\Q[x]/f(x)$ by $K$. Then
\begin{equation*}
|\Delta(K)|\ll \frac{Y^6}{\ind(f)^2}\ll X^{1+6\kappa-2\delta_1}.
\end{equation*}
Therefore, denoting the number of divisors of $m$ by $\sigma_0(m)$, we have
\begin{equation}\label{eqse1}
\begin{array}{rcl}
\displaystyle\sum_{\substack{n,d\geq 1\\nd> X^{1/4+\delta_1}}}
\#\{f\in \Sigma_{dn}^\irr\,:\,h(f)<Y\}
&\ll&
\displaystyle \sum_{m>X^{1/4+\delta_1}}\sigma_0(m)^2
\#\{f\in \Sigma_m^\irr\,:\,h(f)<Y)\\[.2in]
&\ll&
\displaystyle\sum_{\substack{[K:\Q]=3\\|\Delta(K)|\ll X^{1+6\kappa-2\delta_1}}}\#S_K(Y)\\[.35in]
&\ll_\epsilon&X^{1-2\delta_1+6\kappa+\epsilon},
\end{array}
\end{equation}
where the final estimate follows from Theorem \ref{thupbound}.
\end{proof}

For an integer $m$, let $\sq(m)$ denote the product of the prime
powers dividing $m$ to exponent at least~2. Next, we bound the sum
over the terms in the second line of \eqref{eqsm}, where $\sq(dn)$ is
large.

\begin{lemma}\label{lemsmsq}
We have
\begin{equation*}
\sum_{\substack{n,d\geq 1\\nd\leq X^{1/4+\delta_1}\\{\rm sq}(dn)>X^{\delta_2}}}
\#\{f\in \Sigma_{dn}^\irr\,:\,h(f)<Y,|\Delta(f)|<n^2X\}
\ll_\epsilon
X^{1+2\kappa+2\delta_1-\delta_2/2+\epsilon} + X^{1+2\kappa-\delta_2/9+\epsilon}.
\end{equation*}
\end{lemma}
\begin{proof}
Applying Theorem \ref{thbound}, we see that the left hand side of the
above equation is bounded by

\begin{align*}
&\ll_\epsilon X^{\epsilon}\displaystyle\sum_{\substack{m\leq X^{1/4+\delta_1}\\{\rm sq}(m)>X^{\delta_2}}}\Bigl( mY^2+\nu(\Sigma_m)(Y^2+m)\min\{Y^3,mX^{1/2}\}\Bigr)\\
&\ll X^{1+2\kappa+2\delta_1-\delta_2/2+\epsilon}+X^{1+2\kappa}\displaystyle\sum_{\substack{m\leq X^{1/4+\delta_1}\\{\rm sq}(m)>X^{\delta_2}}}\nu(\Sigma_m)m.\\
\end{align*}
Writing $m=rs$, where $r$ is squarefull, $s$ is squarefree, and
$(r,s)=1$, we get
\begin{equation*}
\begin{array}{rcl}
\displaystyle\sum_{\substack{m\leq X^{1/4+\delta_1}\\{\rm sq}(m)>X^{\delta_2}}}\nu(m)m&\ll_\epsilon&
\displaystyle X^{\epsilon}\sum_{r>X^{\delta_2}}\sum_{s\leq X^{1/4+\delta_1}/r}\frac{rs}{r^{5/3}s^{2}}\\[.2in]
&\ll_\epsilon&\displaystyle
X^\epsilon\sum_{r>X^{\delta_1}}r^{-2/3},
\end{array}
\end{equation*}
where the first equality follows from Proposition \ref{propnonmaxpp}.
Now, a squareful number $r$ has a square factor $y^2$ with $y\geq
r^{1/3}$, so we have
\begin{equation*}
\begin{array}{rcl}
\displaystyle\sum_{\substack{m\leq X^{1/4+\delta_1}\\{\rm sq}(m)>X^{\delta_2}}}
\nu(m)m&\ll_\epsilon&
\displaystyle X^\epsilon\sum_{y>X^{\delta_2/3}}y^{-4/3}\\[.25in]
&\ll_\epsilon& \displaystyle X^{-\delta_2/9+\epsilon},
\end{array}
\end{equation*}
concluding the proof of the lemma.
\end{proof}

We are now ready to prove the main result of this section.

\vspace{.1in}
\noindent {\bf Proof of Theorem \ref{theorem3}:} Equation \eqref{eqsm}
and Lemmas \ref{lemtailestimate} and \ref{lemsmsq} imply that we have
\begin{equation*}
\begin{array}{rcl}
\displaystyle\sum_{\substack{[K:\Q=3]\\|\Delta(K)|\leq X}}|S_K(Y)|
&=&\displaystyle
\sum_{\substack{n,d\geq 1\\nd\leq X^{1/4+\delta_1}\\{\rm sq}(dn)\leq
    X^{\delta_2}}} \mu(d)\#\{f\in \Sigma_{dn}^\irr\,:\,h(f)<Y,
|\Delta(f)|<n^2X\}+O_\epsilon(E_1)
\\[.3in]&=&
\displaystyle
\sum_{\substack{n,d\geq 1\\nd\leq X^{1/4+\delta_1}\\{\rm sq}(dn)\leq X^{\delta_2}}}
\mu(d)\#\{f\in \Sigma_{dn}\,:\,h(f)<Y,|\Delta(f)|<n^2X\}
+O_\epsilon(E_1+X^{3/4+3\kappa+\epsilon}),
\end{array}
\end{equation*}
where we use Proposition \ref{propred3} to prove that the number of
reducible elements $\Sigma_{dn}$ is negligible, and where the error
term $E_1$ is defined to be
\begin{equation*}
E_1:=X^{1+6\kappa-2\delta_1+\epsilon}+X^{1+2\kappa+2\delta_1-\delta_2/2+\epsilon}
+ X^{1+2\kappa-\delta_2/9+\epsilon}.
\end{equation*}
Write $nd=m\ell$, where $m$ is squarefree, $\ell$ is squarefull, and
$(m,\ell)=1$. Estimating the number of irreducible elements in
$\Sigma_{dn}$ having bounded height and discriminant using
Theorem~\ref{thbound}, we obtain
%
\begin{equation*}
\{f\in \Sigma_{dn}\,:\,h(f)<Y,|\Delta(f)|<n^2X\}
=\nu(\Sigma_{dn})\Vol(V(\R)_{n^2X\,Y})
+O_\epsilon(Y^{3+\epsilon}\ell/m^{1-\epsilon}+Y^2m^{\epsilon}).
\end{equation*}
Adding the error term in the right hand side of the above equation
over $m\leq X^{1/4+\delta_1}$ and $\ell\leq X^{\delta_2}$, we obtain
the following estimate.
%
\begin{equation*}
\begin{array}{rcl}
\displaystyle\sum_{\substack{[K:\Q=3]\\|\Delta(K)|\leq X}}|S_K(Y)|&=&
\displaystyle
\sum_{\substack{n,d\geq 1\\nd\leq X^{1/4+\delta_1}\\{\rm sq}(dn)\leq X^{\delta_2}}}
\mu(d)\nu(\Sigma_{dn})\Vol(V(\R)_{n^2X\,Y})+O(E(\kappa,\delta_1,\delta_2))
\\[.3in]&=&\displaystyle
\sum_{n\geq 1}\sigma(n)\Vol(V(\R)_{n^2X\,Y})+O(E(\kappa,\delta_1,\delta_2)),
\end{array}
\end{equation*}
where
\begin{equation}\label{eqE}
E(\kappa,\delta_1,\delta_2):=
X^{1+6\kappa-2\delta_1+\epsilon}+X^{1+2\kappa+2\delta_1-\delta_2/2+\epsilon}
+ X^{1+2\kappa-\delta_2/9+\epsilon}
+X^{3/4+3\kappa+2\delta_2+\epsilon}+X^{3/4+2\kappa+\delta_1+\epsilon}.
\end{equation}
Since it is clearly possible to pick positive constants $\kappa$,
$\delta_1$, and $\delta_2$ such that
$E(\kappa,\delta_1,\delta_2)=o(X)$, we recover Theorem \ref{theorem3}.


\end{document}